\documentclass{amsart}
\usepackage{amsrefs}
\usepackage[T1]{fontenc}
\usepackage{graphicx} 
\usepackage{amsfonts, amsmath, amssymb}
\usepackage[english]{babel}
\usepackage{framed}
\usepackage{array}
\numberwithin{equation}{section}
\newtheorem{theorem}{Theorem}
\newtheorem{lemma}{Lemma}
\newtheorem{conj}{Conjecture}

\newcommand{\N}{\mathbb N}

\newcommand{\C}{\mathbb C}

\def\legendre#1#2{\genfrac(){0.2pt}0{#1}{#2}}
\newcommand{\minilegendre}[2]{\genfrac{(}{)}{}{}{#1}{#2}}
\title{Series for $1/\pi$ arising from Cauchy product}
\author{Roman Le Lan}
\email{roman.lelan1@gmail.com}
\begin{document}
\begin{abstract}
   In this note, we evaluate a series for $1/\pi$ conjectured by Sun. Our proof uses the Cauchy product and hypergeometric transformations. From this result, we derive two additional analogous series involving polynomials of degree $3$. Further identities can be proved using our method; these are presented in a table at the end of the note.
\end{abstract}
\maketitle

\section{Introduction}
Ramanujan--Sato type series for $1/\pi$ are widely studied. Sun conjectured many evaluations of such series. In particular, he \cite{cong-pi-sun}*{Conjecture 4} proposed 37 conjectures concerning series of the form
\[
\sum_{n=0}^\infty (an+b)q^n\sum_{k=0}^n\binom{-x}{k}\binom{x-1}{n-k}\binom{-y}{k}\binom{y-1}{n-k}=\frac{\alpha}{\pi},
\]
where $a$ and $b$ are nonzero integers, $q$, $x$, and $y$ are nonzero rational numbers, and $\alpha$ is an algebraic number of degree at most $2$. Almkvist and Aycock \cite{aa} evaluated 36 of these formulas. In this note, we prove the last identity from Sun's list \cite{cong-pi-sun}*{(4.14)} (see also \cite{sunconjres}).
\begin{theorem}
We have the following identity.
    \begin{equation}
    \sum_{n=0}^\infty \frac{3n-1}{2^n}\sum_{k=0}^n\binom{-1/3}{k}\binom{-2/3}{n-k}\binom{-1/6}{k}\binom{-5/6}{n-k}=\frac{3\sqrt6}{2\pi}
    \label{th1}
    \end{equation}
    \label{1}
\end{theorem}
It is natural to consider $p$-adic analogues of series for $1/\pi$. In particular, motivated by \eqref{th1}, Sun \cite{sunconjres} conjectured a supercongruence.
\begin{conj}[Z.-W. Sun]
    Let $p>3$ be a prime. Then,
    \[\sum_{n=0}^{p-1} \frac{3n-1}{2^n}\sum_{k=0}^n\binom{-1/3}{k}\binom{-2/3}{n-k}\binom{-1/6}{k}\binom{-5/6}{n-k}\equiv-p\legendre{-6}{p}\pmod{p^2}.\]
\end{conj}
Here, $\minilegendre{q}{p}$ denotes the Legendre symbol. We were unable to prove this congruence. 

Our second result relies on Theorem \ref{1}. We use the method given by Sun \cite{newtypeserie}.

\begin{theorem}
We have the following series.
\begin{equation}
\sum_{n=0}^\infty \frac{9n^3-27n^2-14n-2}{2^n}\sum_{k=0}^n\binom{-1/3}{k}\binom{-2/3}{n-k}\binom{-1/6}{k}\binom{-5/6}{n-k}=-\frac{3\sqrt6}{8\pi}
\label{th2}
\end{equation}
\begin{equation}
\sum_{n=0}^\infty \frac{18n^3-54n^2-25n-5}{2^n}\sum_{k=0}^n\binom{-1/3}{k}\binom{-2/3}{n-k}\binom{-1/6}{k}\binom{-5/6}{n-k}=\frac{3\sqrt6}{4\pi}
\label{th2.2}
\end{equation}
\label{2}
\end{theorem}

Using the same method as for Theorem \ref{1}, we are able to prove eight additional series of the form
\begin{equation}
\sum_{n=0}^\infty (an+b)q^n\sum_{k=0}^n\binom{-1/3}{k}\binom{-2/3}{n-k}\binom{-1/6}{k}\binom{-5/6}{n-k}=\frac{\alpha}{\pi},
\label{general form}
\end{equation}
where $a$ and $b$ are integers, $q$ is rational, and $\alpha$ is an algebraic number of degree at most $2$. They are listed at the end of the note in Table \ref{table}.

This note uses standard notation. The Pochhammer symbol is given by $(a)_n=\prod_{k=0}^{n-1}(a+k)$.
\section{Preliminary results}
To prove Theorems \ref{1} and \ref{2}, we need several intermediate results. The first lemma establishes a useful transformation.
\begin{lemma}
    Let $n\ge0$ be an integer. Then
    \begin{equation}
    108^n\sum_{k=0}^n\binom{-1/3}{k}\binom{-1/6}{k}\binom{-1/3}{n-k}\binom{-1/6}{n-k}=\binom{2n}{n}^2\binom{3n}{n}
    \label{lem1}
    \end{equation}
\end{lemma}
\begin{proof}
    It is easy to see that \eqref{lem1} holds for $n=0$ and $n=1$.
    
Let $u_n$ denote the left-hand side of \eqref{lem1}. Using Zeilberger's algorithm on $u_n$, we obtain the recurrence relation
\[
(n+1)^3 u_{n+1} - 6(3n+2)(2n+1)(3n+1) u_n = 0.
\]
Moreover, one can easily verify that $\binom{2n}{n}^2 \binom{3n}{n}$ also satisfies this recurrence, which completes the proof of the lemma.
\end{proof}
The following lemma presents a Ramanujan-type formula due to Guillera \cite{guillera}*{Table 3}.
\begin{lemma}[J. Guillera]
We have the following Ramanujan-type series.
    \begin{equation}
\sum_{k=0}^\infty\biggl(\frac{6k}{3\sqrt3}+\frac1{3\sqrt{3}}\biggr)\frac{\bigl(\frac{1}{2}\bigr)_k\bigl(\frac{1}{3}\bigr)_k\bigl(\frac{2}{3}\bigr)_k}{2^kk!^3}=\frac{1}{\pi}
    \label{ramanujan}
    \end{equation}
\end{lemma}
The other identities given by Guillera \cite{guillera}*{Table 3} are used to prove the formulas in Table \ref{table}.
\section{Proof of Theorem \ref{1}}

We can now prove Theorem \ref{1}.

\begin{proof}[Proof of Theorem \ref{1}]
Let $z\in\C$ such that $|z|<1$ and let $P$ be a function defined for all $z$ by
    \begin{equation}
    P(z)=\sum_{n=0}^\infty z^n\sum_{k=0}^n\binom{-1/3}{k}\binom{-2/3}{n-k}\binom{-1/6}{k}\binom{-5/6}{n-k}.
    \label{P(z)}
    \end{equation}
    By the Cauchy product of two series, we obtain a new expression for $P$.
    \[P(z)=\Biggr(\sum_{k=0}^\infty z^k\binom{-1/3}{k}\binom{-1/6}{k}\Biggl)\Biggl(\sum_{k=0}^\infty z^k\binom{-2/3}{k}\binom{-5/6}{k}\Biggr)\]
    The series $P$ can thus be expressed as a product of two hypergeometric functions ${}_2F_1$. Using Euler's transformation, we obtain
    \begin{align*}
        P(z)&={}_2F_1\biggl(\frac13,\frac16;1;z\biggr){}_2F_1\biggl(\frac23,\frac56;1;z\biggr)\\
        &=(1-z)^{-\frac12}{}_2F_1\biggl(\frac13,\frac16;1;z\biggr)^2.
    \end{align*}
    Then, upon applying another Cauchy product, it follows that
\[
P(z) = (1-z)^{-\frac12} \sum_{n=0}^\infty z^n \sum_{k=0}^n \binom{-1/3}{k} \binom{-1/6}{k} \binom{-1/3}{n-k} \binom{-1/6}{n-k}.
\]
Using identity \eqref{lem1} from Lemma 1, this becomes
    \begin{equation}
    P(z)=(1-z)^{-\frac12}\sum_{n=0}^\infty \biggl(\frac{z}{108}\biggr)^n\binom{2n}{n}^2\binom{3n}{n}.
    \label{simp. P(z)}
    \end{equation}
    Let us apply the operator
    \[-1+3z\frac{d}{dz}\Bigg|_{z=\frac12}\]
    to both sides of \eqref{P(z)} to obtain the relation
\[
\frac{3}{2} P'\biggl(\frac12\biggr) - P\biggl(\frac12\biggr) = \sum_{n=0}^\infty \frac{3n-1}{2^n} \sum_{k=0}^n \binom{-1/3}{k} \binom{-2/3}{n-k} \binom{-1/6}{k} \binom{-5/6}{n-k},
\]
where the $'$ stands for the derivative with respect to $z$. Hence, using equation \eqref{simp. P(z)} in the above relation, we have
    \begin{align*}
        \sum_{n=0}^\infty \frac{3n-1}{2^n}\sum_{k=0}^n&\binom{-1/3}{k}\binom{-2/3}{n-k}\binom{-1/6}{k}\binom{-5/6}{n-k}\\
        &=\frac{\sqrt{2}}{2}\sum_{n=0}^\infty \frac{1}{216^n}\binom{2n}{n}^2\binom{3n}{n}+3\sqrt2\sum_{n=0}^\infty \frac{n}{216^n}\binom{2n}{n}^2\binom{3n}{n}\\
        &=\frac{\sqrt{2}}{2}\sum_{k=0}^{\infty}(6k+1)\frac{\bigl(\frac{1}{2}\bigr)_k\bigl(\frac{1}{3}\bigr)_k\bigl(\frac{2}{3}\bigr)_k}{2^kk!^3}.
        \label{2.3}
    \end{align*}
Identity \eqref{ramanujan} from Lemma 2 completes the proof.
 \begin{align*}
     \sum_{n=0}^\infty \frac{3n-1}{2^n}\sum_{k=0}^n\binom{-1/3}{k}\binom{-2/3}{n-k}\binom{-1/6}{k}\binom{-5/6}{n-k}&=\frac{\sqrt{2}}{2}\times\frac{3\sqrt{3}}{\pi}\\&=\frac{3\sqrt{6}}{2\pi}
 \end{align*}
 Theorem \ref{1} is proved.
\end{proof}
\section{Proof of Theorem \ref{2}}
The proof of these results is enabled by that of Theorem \ref{1}. We then use the method given by Sun \cite{newtypeserie} to produce series of \textit{type S}.
\begin{proof}[Proof of Theorem \ref{2}]
Let $n\in\N$. Set
\[
a_n = \frac{1}{2^n} \sum_{k=0}^n \binom{-1/3}{k} \binom{-2/3}{n-k} \binom{-1/6}{k} \binom{-5/6}{n-k}.
\]
Using Zeilberger's algorithm on $a_n$, we obtain the following recurrence relation.
\[144(n+2)^3a_{n+2}-4(2n+3)(18n^2+54n+47)a_{n+1}+(n+1)(6n+5)(6n+7)a_n=0\]
Hence,
\begin{align*}
    0&=\sum_{n=0}^\infty(n+1)(6n+5)(6n+7)a_n-4\sum_{n=0}^\infty(2n+3)(18n^2+54n+47)a_{n+1}\\&\quad+144\sum_{n=0}^\infty(n+2)^3a_{n+2}\\
    &=\sum_{k=0}^\infty(k+1)(6k+5)(6k+7)a_k-4\sum_{k=1}^\infty(2k+1)(18k^2+18k+11)a_k\\&\quad+144\sum_{k=2}^\infty k^3a_k\\
    &=\sum_{k=0}^\infty\bigl((k+1)(6k+5)(6k+7)-4(2k+1)(18k^2+18k+11)+144k^3 \bigr)a_k\\&\quad+44a_0-144a_1\\
    &=\sum_{k=0}^\infty(36k^3-108k^2-53k-9)a_k.
\end{align*}
By subtracting or adding this to \eqref{th1}, we obtain the identities \eqref{th2} and \eqref{th2.2} as desired. Theorem \ref{2} is proved.
\end{proof}
\section{Concluding remarks}
The method we used to prove Theorem 2 can be applied to the series in Table \ref{table} to find other formulas for $1/\pi$.

Sun \cite{cong-pi-sun} noted that his conjectural identity \cite{cong-pi-sun}*{(4.22)} would imply the following formula.
\begin{equation}
    \sum_{n=0}^\infty \frac{16854n+985}{(-250000)^n}\sum_{k=0}^n \binom{-1/3}{k} \binom{-2/3}{n-k} \binom{-1/6}{k} \binom{-5/6}{n-k}=\frac{4500000}{89\sqrt{267}\pi}
    \label{sunerror}
\end{equation}
However, this statement is incorrect. Indeed, we correct the left-hand side of \eqref{sunerror} to 
\[\sum_{n=0}^\infty \frac{16854n+985}{(-250000)^n}\sum_{k=0}^n \binom{-1/3}{k} \binom{-2/3}{n-k} \binom{-1/6}{k} \binom{-5/6}{n-k}=\frac{750000\sqrt{267}}{7921\pi};\]
see Table \ref{table}. 
\section*{Acknowledgments}
The author thanks Prof. Ji-Cai Liu for his valuable advice and encouragement.
\bibliographystyle{amsplain}
\bibliography{bibly}
\setlength{\extrarowheight}{0.3cm}
\begin{table}[ht!]
\centering
\begin{tabular}{|c c c c|c c c c|} 
 \hline
$a$ & $b$ & $q$ & $\alpha$ & $a$ & $b$ & $q$ & $\alpha$\\ [0.5ex] 
 \hline\hline
 $3$ & $-1$ & $\frac{1}{2}$ & $\frac{3\sqrt{6}}{2}$ & $18150$ & $1433$ & $-\frac{1}{3024}$ & $2592\sqrt3$\\ 
 $50$ & $19$ & $-\frac{9}{16}$ & $\frac{32\sqrt3}{3}$ & $16854$ & $985$ & $-\frac{1}{500^2}$ & $\frac{750000\sqrt{267}}{7921}$ \\
 $17$ & $1$ & $-\frac{1}{16}$ & $\frac{16\sqrt{51}}{17}$ & $75$ & $7$ & $\frac{2}{27}$ & $\frac{81\sqrt3}4$\\
 $162$ & $19$ & $-\frac{1}{80}$ & $32\sqrt3$ & $363$ & $38$ & $\frac{4}{125}$ & $\frac{17\sqrt{15}}{2}$\\
 $150$ & $13$ & $-\frac{1}{1024}$ & $\frac{6144\sqrt{123}}{1681}$ &  &  &  &  \\ [1ex] 
 \hline
\end{tabular}
\caption{Other series of the form \eqref{general form}.}
\label{table}
\end{table}
\end{document}